\documentclass[a4paper]{amsart}

\usepackage{cite}

\title[Fractional and Lorentz space estimates]{Limiting fractional and Lorentz spaces estimates of differential forms}
\author{Jean Van Schaftingen}
\address{Universit\'e Catholique de Louvain\\ D\'epartement de Math\'ematique\\ Chemin du Cyclotron 2\\ 1348 Louvain-la-Neuve \\ Belgium}
\email{Jean.VanSchaftingen@uclouvain.be}
\thanks{The author is supported by the Fonds de la Recherche Scientifique--FNRS}
\newtheorem{theorem}{Theorem}
\newtheorem{proposition}{Proposition}[section]
\newtheorem{corollary}[proposition]{Corollary}
\newtheorem{lemma}[proposition]{Lemma}
\theoremstyle{definition}

\newtheorem{openproblem}{Open problem}
\theoremstyle{remark}

\newcommand{\R}{\mathbf{R}}

\newcommand{\Norm}[1]{\Vert #1  \Vert}
\newcommand{\abs}[1]{\lvert #1  \rvert}


\begin{document}

\maketitle

\begin{abstract}
We obtain estimates in Besov, Lizorkin-Triebel and Lorentz spaces of differential forms on $\R^n$ in terms of their $L^1$ norm.
\end{abstract}

\section{Introduction}

The classical Hodge theory states that if $u \in C^\infty_c(\R^N; \bigwedge^\ell \R^n)$, if $1 < p < \infty$, one has 
\begin{equation}
\label{LpEstimate}
 \Norm{Du}_{L^p} \le C(\Norm{du}_{\mathrm{L}^p} + \Norm{\delta u}_{\mathrm{L}^p})
\end{equation}
where $du$ is the exterior differential and $\delta u$ the exterior codifferential.
This estimate is known to fail when $p=1$ or $p=\infty$.

When $p=1$, J. Bourgain and H. Brezis \cite{BB2004,BB2007}, and L. Lanzani and E. Stein \cite{LS2005} have obtained for $2  \le \ell \le n-2$ the estimate
\[
 \Norm{u}_{L^{n/(n-1)}} \le C(\Norm{du}_{\mathrm{L}^1} + \Norm{\delta u}_{\mathrm{L}^1}),
\]
which would be the consequence that would follow by the Sobolev embedding from \eqref{LpEstimate} with $p=1$.
When $\ell=1$ or $\ell=n-1$ one has to assume that $du$ or $\delta u$ vanishes.

I. Mitrea and M. Mitrea \cite{MitreaMitrea} have in a recent work extended these estimates to homogeneous Besov spaces. Using interpolation theory, they could replace the norm $\Norm{u}_{L^{n/(n-1)}}$ by $\Norm{u}_{\dot{B}^{s}_{p,q}}$ with $\frac{1}{p}-\frac{s}{n}=1-\frac{1}{n}$ and $q = \frac{2}{1-s}$. The goal of the present paper is to improve the assumption on $q$ by relying on previous results and methods.

\medskip

We follow H. Triebel \cite{Triebel1983} for the definitions of the function spaces.
The first result is the estimate for the Besov spaces $\dot{B}^s_{p,q}(\R^n)$:

\begin{theorem}
\label{thmBesov}
For every $s \in (0, 1)$, $p>1$ and $q > 1$, if
\begin{equation}
\label{condition}
 \frac{1}{p}-\frac{s}{n}=1-\frac{1}{n},
\end{equation}
then there exists $C > 0$ such that for every $u \in C^\infty_c(\R^n; \bigwedge^\ell \R^n)$, with moreover, $\delta u=0$ if $\ell=1$ and $du=0$ if $\ell=n-1$, one has
\[
 \Norm{u}_{\dot{B}^{s}_{p, q}} \le C (\Norm{du}_{L^1}+\Norm{\delta u}_{L^1}).
\]
\end{theorem}

In particular, since $\Norm{u}_{\dot{W}^{s,p}} =\Norm{u}_{\dot{B}^{s}_{p, p}}$, one has the estimate
\begin{equation}
 \Norm{u}_{\dot{W}^{s,p}} \le C (\Norm{du}_{L^1}+\Norm{\delta u}_{L^1}).
\end{equation}
In Theorem~\ref{thmBesov}, we assume that $q > 1$. If it held for some $q \in (0,1]$, then the embedding of $F^1_{1, 2}(\R^n) \subset \dot{B}^{0}_{n/(n-1),q}(\R^n)$ would hold. This can only be the case when $q\ge 1$. Therefore, the only possible improvement of Theorem~\ref{thmBesov} would be the limiting case $q=1$:

\begin{openproblem}
 Does Theorem \ref{thmBesov} hold for $q=1$?
\end{openproblem}

The estimate of Theorem~\ref{thmBesov} follows from the corresponding estimate for homogeneous Lizorkin--Triebel spaces $\dot{F}^s_{p,q}(\R^n)$:

\begin{theorem}
\label{thmLizorkinTriebel}
For every $s \in (0, 1)$, $p>1$ and $q > 0$, if \eqref{condition} holds, then
there exists $C > 0$ such that for every $u \in C^\infty_c(\R^n; \bigwedge^\ell \R^n)$, with moreover, $\delta u=0$ if $\ell=1$ and $du=0$ if $\ell=n-1$, one has
\[
 \Norm{u}_{\dot{F}^s_{p,q}} \le C (\Norm{du}_{L^1}+\Norm{\delta u}_{L^1}).
\]
\end{theorem}

Note that here there is no restriction on $q > 0$. Finally, the latter estimate has an interesting consequence for Lorentz spaces.

\begin{theorem}
\label{thmLorentz}
For every $q > 1$, then there exists $C > 0$ such that for every $u \in C^\infty_c(\R^n; \bigwedge^\ell \R^n)$, with moreover, $\delta u=0$ if $\ell=1$ and $du=0$ if $\ell=n-1$, one has  
\[
 \Norm{u}_{L^{\frac{n}{n-1},q}} \le C (\Norm{du}_{L^1}+\Norm{\delta u}_{L^1}).
\]
\end{theorem}

In Theorem~\ref{thmLorentz} the case $q=1$ and $\ell=0$ is equivalent with the embedding of $W^{1,1}(\R^n)$ in $L^{\frac{n}{n-1},1}(\R^n)$ which was obtained by  J. Peetre \cite{Peetre} (see also \cite{Z1989}). This raises the question

\begin{openproblem}
Does Theorem~\ref{thmLorentz} hold for $q=1$ and $\ell \ge 1$?
\end{openproblem}

The proof of the theorems rely on the techniques developed by the author \cite{VS2004Curves, VS2004Divf}, and on classical embeddings and regularity theory in fractional spaces.



\section{The main tool}

Our main tool is a generalization of an estimate for  divergence-free $L^1$ vector fields of the author \cite{VS2004Divf}:

\begin{proposition}
\label{propositionMainEstimate}
For every $s \in (0,1)$, $p > 1$ and $q > 0$ with $sp=n$, there exists $C > 0$ such that for every $f \in (C^\infty_c\cap L^1)(\R^n; \bigwedge^{n-1} \R^n)$ and $\varphi \in C^\infty_c(\R^n; \bigwedge\R^n)$, if $df=0$, 
\[
  \int_{\R^n} f \wedge \varphi \le C \Norm{f}_{L^1} \Norm{\varphi}_{\dot{F}^{s}_{p,q}}.
\]
\end{proposition}

The proof of this proposition follows the method introduced by the author \cite{CVS, VS2004Curves,VS2004Divf,VS2008} and followed subsequently by L. Lanzani and E. Stein \cite{LS2005} and I.~Mitrea and M.~Mitrea \cite{MitreaMitrea}. The extension to the case $q=p$ in a previous work of the author \cite[Remark 5]{VS2004Divf} (see also \cite[Remark 2]{VS2008} and \cite{CVS}); the proposition can be deduced therefrom by following a remark in a subsequent paper \cite[Remark 4.2]{VS2006BMO}.

\begin{proof}
Write $\varphi=\varphi^1dx_1+\varphi^n dx_n$ and $f=f_1 dx_2 \wedge \dotsm \wedge dx_n+\dotsb+f_n dx_1 \wedge \dotsm \wedge dx_{n-1}$. Without loss of generality, we shall estimate
\[
  \int_{\R^n} f_1 \varphi^1.
\]
Fix $t \in \R$, and consider the function $\psi : \R^{n-1} \to \R$ defined by $\psi(y)=\varphi^1(t,y)$. Choose $\rho \in C^\infty_c(\R^n)$ such that $\int_{\R^{n-1}} \rho =1$ and set $\rho_\varepsilon(y)=\frac{1}{\varepsilon^{n-1}} \rho (\frac{x}{\varepsilon})$. For every $\alpha \in (0,1)$, there is a constant $C > 0$ that only depends on $\rho$ and $\alpha$ such that 
\[
 \Norm{\nabla \rho_\varepsilon \ast \psi}_{L^\infty} \le C \varepsilon^{\alpha-1}\abs{\psi}_{C^{0, \alpha}(\R^{n-1})} 
\]
and
\[
 \Norm{\psi-\rho_\varepsilon \ast \psi}_{L^\infty} \le C \varepsilon^\alpha \abs{\psi}_{C^{0, \alpha}(\R^{n-1})},
\]
where $\abs{\psi}_{C^{0, \alpha}(\R^{n-1})}$ is the $C^{0, \alpha}$ seminorm of $\psi$, i.e.,
\[
 \abs{\psi}_{C^{0, \alpha}(\R^{n-1})}=\sup_{y, z \in \R^{n-1}} \frac{\abs{\psi(z)-\psi(y)}}{\abs{z-y}}.
\]

One has on the one hand
\[
 \int_{\R^{n-1}} f_1(t, \cdot) (\psi-\rho_\varepsilon \ast  v) \le C\Norm{f_1(t, \cdot)}_{L^1(\R^{n-1})} \varepsilon^\alpha \abs{\psi}_{C^{0, \alpha}(\R^{n-1})}.
\]
On the other hand, by integration by parts, and since $\sum_{i=1}^n \partial_i f_i=0$,
\[
\begin{split}
 \int_{\R^{n-1}} f_1(t, \cdot) \rho_\varepsilon \ast \psi&=-\sum_{i=2}^{n} \int_{\R^{n-1}} \int_{\R^+} f_i(t,y) \partial_i (\rho_\varepsilon \ast \psi)(y)\,dt\,dy  \\
&\le C\Norm{f_1(t, \cdot)}_{L^1(\R^{n-1})} \varepsilon^{\alpha-1} \abs{\psi}_{C^{0, \alpha}(\R^{n-1})}.
\end{split}
\]
Taking $\varepsilon=\Norm{f}_{L^1(\R^n)}/\Norm{f(t, \cdot)}_{L^1(\R^{n-1})}$, one obtains
\begin{equation}
\label{ineqHolderSpace}
 \int_{\R^{n-1}} f_1 \psi \le C \Norm{f}^{\alpha}_{L^1(\R^n)} \Norm{f_1(t, \cdot)}^{1-\alpha}_{L^1(\R^{n-1})} \abs{\psi}_{C^{0, \alpha}(\R^{n-1})}.
\end{equation}

Now, by the embedding theorem for Lizorkin--Triebel spaces, one has the estimate
\[
 \abs{\psi}_{C^{0, \alpha}} \le C \Norm{ \psi}_{\dot{F}^{s}_{p,q}(\R^{n-1})};
\]
with $\alpha=\frac{1}{p}$;
hence from \eqref{ineqHolderSpace} we deduce the inequality
\[
  \int_{\R^{n-1}} f_1 \psi \le C \Norm{f}_{L^1}^{\frac{1}{p}} \Norm{f_1(\cdot, t)}_{L^1}^{1-\frac{1}{p}} \Norm{\psi}_{\dot{F}^{s}_{p,q}(\R^{n-1})}.
\]

Now, recalling that, as a direct consequence of the Fubini property that is stated in \cite[Theorem 2.5.13]{Triebel1983} \cite[Th\'eor\`eme 2]{Bourdaud}, \cite[Theorem 2.3.4/2]{RS}
\[
 \Bigl(\int_{\R} \Norm{\varphi(t, \cdot)}_{\dot{F}^{s, p} (\R^{n-1})}^p \,dt\Bigr) \le C \Norm{\varphi}_{\dot{F}^{s, p}(\R^{n})}^p,
\]
one concludes, using H\"older's inequality that 
\[
 \int_{\R^n} f_1 u^1\le C \Norm{f}_{L^1}^{\frac{1}{p}}\int_{\R} \bigl(\Norm{f_1(\cdot, t)}_{L^1}^{1-\frac{1}{p}} \Norm{\varphi(t, \cdot)}_{\dot{F}^{s, p}(\R^{n-1})}\bigr)\,dt 
 \le C' \Norm{f}_{L^1} \Norm{\varphi}_{\dot{F}^{s, p}(\R^{n})}.\qedhere
\]
\end{proof}

\begin{proposition}
\label{propositionNegativeLizorkinTriebel}
For every $s \in (0,1)$, $p > 1$ with $\frac{1}{p}+\frac{s}{n}=1$, $q > 1$ and $1 \le \ell \le n-1$, there exists $C > 0$ such that for every $f \in C^\infty_c(\R^n; \bigwedge^\ell \R^n)$ with $df = 0$, one has
\[
  \Norm{f}_{\dot{F}^{-s,p}_q} \le C \Norm{f}_{L^1}.
\]
\end{proposition}
\begin{proof}
The proposition will be proved by downward induction.
The proposition is true  for $\ell=n-1$ by Proposition~\ref{propositionMainEstimate}. Assume now that it holds for $\ell + 1$, and let $f \in C^\infty_c(\R^n; \bigwedge^\ell \R^n)$. Since $d(f \wedge dx_i)=0$, Proposition~\ref{propositionMainEstimate} is applicable and 
\[
 \Norm{f}_{\dot{F}^{-s,\frac{n}{n-s}}_q} \le \sum_{i=1}^n \Norm{f \wedge dx_i}_{\dot{F}^{-s,\frac{n}{n-s}}_q} \le C\sum_{i=1}^n \Norm{f}_{L^1}=Cn \Norm{f}_{L^1}.\qedhere
\]
\end{proof}

A useful corollary of the previous proposition is 

\begin{corollary}
For every $s \in (0,1)$, $p > 1$ with $\frac{1}{p}+\frac{s}{n}=1$, $q > 1$ and $1 \le \ell \le n-1$, there exists $C > 0$ such that for every $f \in C^\infty_c(\R^n; \bigwedge^\ell \R^n)$ with $df = 0$, one has
\[
  \Norm{f}_{\dot{B}^{-s,p}_q} \le C \Norm{f}_{L^1}.
\]
\end{corollary}
\begin{proof}
This follows from classical embeddings between Besov and Lizorkin--Triebel spaces; see the proof of Theorem~\ref{thmBesov} below.
\end{proof}

\section{Proofs of the main results}

We begin by proving Theorem~\ref{thmLizorkinTriebel}:

\begin{proof}[Proof of Theorem~\ref{thmLizorkinTriebel}]
To fix ideas, assume that $2 \le \ell \le n-1$. 
Recall that one has
\[
  u=d \bigl(K \ast (\delta u)\bigr)+\delta \bigl(K \ast (d u)\bigr),
\]
where the Newton kernel is defined by $K(x)=\frac{\Gamma(\frac{n}{2})}{2 \pi^{\frac{n}{2}} \abs{x}^{n-2}}$. By the classical elliptic estimates for Lizorkin--Triebel spaces,
\begin{align*}
  \Norm{K \ast (\delta u)}_{\dot{F}^{s+1}_{p,q}} &\le C \Norm{\delta u}_{\dot{F}^{s-1}_{p,q}}
&&\text{and}
 & \Norm{K \ast (\delta u)}_{\dot{F}^{s+1}_{p,q}} &\le C \Norm{\delta u}_{\dot{F}^{s-1}_{p,q}}.
\end{align*}
Now, since $d(du)=0$, Proposition~\ref{propositionNegativeLizorkinTriebel} is applicable and yields
\[
   \Norm{K \ast (d u)}_{\dot{F}^{s+1}_{p,q}} \le C \Norm{\delta u}_{L^1}.
\] 
Since $\delta(\delta u)=0$, one can by the Hodge duality between $d$ and $\delta$ treat $\Norm{K \ast (d u)}_{\dot{F}^{s+1}_{p,q}}$ similarly.
\end{proof}

We can now deduce Theorem~\ref{thmBesov} from Theorem~\ref{thmLizorkinTriebel}:

\begin{proof}[Proof of Theorem~\ref{thmBesov}]
First assume that $q \ge p$. Then one has
\[
 \Norm{u}_{\dot{B}^s_{p,q}} \le C \Norm{u}_{\dot{F}^s_{p,q}},
\]
and Theorem~\ref{thmBesov} follows from Theorem~\ref{thmLizorkinTriebel}.
Otherwise, if $q < p$, then by the embedding theorems of Besov spaces, 
\[
 \Norm{u}_{\dot{B}^s_{p,q}} \le C\Norm{u}_{\dot{B}^r_{q,q}}=C\Norm{u}_{\dot{F}^r_{q,q}}
\]
with $r=s+n(\frac{1}{q}-\frac{1}{p})$ and Theorem~\ref{thmBesov} also follows from Theorem~\ref{thmLizorkinTriebel}.
\end{proof}

We finish with the proof of Theorem~\ref{thmLorentz}. It relies on the 
\begin{lemma}
\label{lemma}
For every $s > 0$, $p> 1$ and $q > 1$ with $sq < n$ and
\begin{equation}
\label{equationS}
 \frac{1}{p}=\frac{1}{q}-\frac{s}{n},
\end{equation} 
there exists $C > 0$ such that for every $u \in C^\infty_c(\R^n)$,
\[
  \Norm{u}_{L^{p,q}}\le C \Norm{u}_{\dot{F}^s_{q,2}}.
\]
\end{lemma}
\begin{proof}
One has
\[
 u= I_{s} *((-\Delta)^\frac{s}{2} u),
\]
where the Riesz kernel $I_s$ is defined for $x \in \R^n$ by
\[
 I_s(x)=\frac{\Gamma(\frac{n-\alpha}{2})}{\pi^{\frac{n}{2}}2^s \Gamma(\frac{s}{2})\abs{x}^{n-s}}.
\]
One has then by Sobolev inequality for Riesz potentials in Lorentz spaces of R.~O'Neil \cite{ONeil1963} (see also e.g. \cite[Theorem 2.10.2]{Z1989}),
\[
 \Norm{u}_{L^{r,p}} \le C \Norm{ (-\Delta)^\frac{s}{2} u}_{L^p}.
\]
One concludes by noting that $\Norm{ (-\Delta)^\frac{s}{2} u}_{L^p}$ and $\Norm{u}_{\dot{F}^{s}_{p,2}}$ are equivalent norms \cite[Theorem 2.3.8 and section 5.2.3]{Triebel1983}.
%
%
%
%
%
%
\end{proof}

\begin{proof}[Proof of Theorem~\ref{thmLorentz}]
Choose $s$ so that \eqref{equationS} holds with $p=\frac{n}{n-1}$. Since
$\frac{1}{q}-\frac{s}{n}=1-\frac{1}{n}$, one can combine Theorem~\ref{thmLizorkinTriebel} and Lemma~\ref{lemma} to obtain the conclusion.
\end{proof}


%
%
%
%
%
%
%

\providecommand{\bysame}{\leavevmode\hbox to3em{\hrulefill}\thinspace}
\providecommand{\MR}{\relax\ifhmode\unskip\space\fi MR }
\providecommand{\MRhref}[2]{%
  \href{http://www.ams.org/mathscinet-getitem?mr=#1}{#2}
}
\providecommand{\href}[2]{#2}

\end{document}